\documentclass[a4wide,12pt,leqno]{article} %"leqno" means it shows the equation number in left.
\usepackage{amsmath,latexsym, amsfonts, amssymb}
\usepackage{mathrsfs}
\usepackage{fullpage}
\usepackage{enumerate}
\usepackage{amssymb, amsmath, amsfonts}
\usepackage{mathrsfs}
\usepackage{epsfig,amsbsy,amsthm} % "amsbsy" produce bold math symbol
\usepackage{float,epsfig}
\usepackage{fixmath}
\usepackage{graphics}
\usepackage{pgfpages}
\usepackage{appendix}
\numberwithin{equation}{section}  %It gives the eq. no in fraction format.
\usepackage{hyperref} % this gives link
\usepackage{graphicx}% the above two gives link
\usepackage{pst-all}
\usepackage{calligra}
\usepackage{color}
\usepackage{tikz}
\DeclareMathAlphabet{\mathpzc}{OT1}{pzc}{m}{it}
\DeclareMathAlphabet{\mathcalligra}{T1}{calligra}{m}{n}
\begin{document}
\newtheorem{theorem}{\bf Theorem}[section]
\newtheorem{proposition}[theorem]{\bf Proposition}
\newtheorem{definition}{\bf Definition}[section]
\newtheorem{corollary}[theorem]{\bf Corollary}
\newtheorem{exam}[theorem]{\bf Example}
\newtheorem{remark}[theorem]{\bf Remark}
\newtheorem{lemma}[theorem]{\bf Lemma}
\newtheorem{assum}[theorem]{\bf Assumption}

%%%%%%%%%%%%%%%%%%%%%%%%%%%%%%%%%%%%%%%%%%%%%%
\newcommand{\be}{\begin{equation}}
\newcommand{\ee}{\end{equation}}
\newcommand{\beno}{\begin{equation*}}
\newcommand{\eeno}{\end{equation*}}
\newcommand{\ba}{\begin{align}}
\newcommand{\ea}{\end{align}}
\newcommand{\bano}{\begin{align*}}
\newcommand{\eano}{\end{align*}}
\newcommand{\bea}{\begin{eqnarray}}
\newcommand{\eea}{\end{eqnarray}}
\newcommand{\beano}{\begin{eqnarray*}}
\newcommand{\eeano}{\end{eqnarray*}}
\newcommand{\tnorm}[1]{|\!|\!| {#1} |\!|\!|}
\renewcommand{\thefootnote}{\arabic{footnote}}
\def \noin{\noindent}
%%%%%%%%%%%%%%%%%%%%%%%%%%% new commands %%%%%%%%%
%%%%%%%%%%%%%% mathbb  %%%%%%%%%%%%%%%%%%%%%%%%%%
\def \N{{\mathbb N}}
\def \R{{\mathbb R}}
\def \S{{\mathbb S}}
\def \V{{\mathbb V}}
\def \Z{{\mathbb Z}}

%%%%%%%%%%%%%%%%%%%nmathcal %%%%%%%%%%%%%%%%%%%%%%%%%
\def \Ac{{\mathcal A}}
\def \Bc{{\mathcal B}}
\def \Cc{{\mathcal C}}
\def \Dc{{\mathcal D}}
\def \Ec{{\mathcal E}}
\def \Fc{{\mathcal F}}
\def \Gc{{\mathcal G}}
\def \Ic{{\mathcal I}}
\def \Jc{{\mathcal J}}
\def \Lc{{\mathcal L}}
\def \Mc{{\mathcal M}}
\def \Nc{{\mathcal N}}
\def \Oc{{\mathcal O}}
\def \Pc{{\mathcal P}}
\def \Qc{{\mathcal Q}}
\def \Rc{{\mathcal R}}
\def \Tc{{\mathcal T}}
\def \Uc{{\mathcal U}}
\def \Wc{{\mathcal W}}
\def \Xc{{\mathcal X}}
\def \Yc{{\mathcal Y}}
\def \Zc{{\mathcal Z}}

%%%%%%%%%%%%%%%%%%%%%%% mathscr %%%%%%%%%%%%%%%%%%%%%%%%%
\def \Cs{\mathscr{C}}
\def \Ds{\mathscr{D}}
\def \Es{\mathscr{E}}
\def \Js{\mathscr{J}}
\def \Ms{\mathscr{M}}
\def \Os{\mathscr{O}}
\def \Qs{\mathscr{Q}}
\def \Rs{\mathscr{R}}
\def \Ss{\mathscr{S}}
\def \Ts{\mathscr{T}}
%%%%%%%%%%%%%%%%%%%%%%%% mathrm %%%%%%%%%%%%%%%%%%%%%%%%%
\def \B {\mathrm{BDF}}
\def \el {\mathrm{el}}
\def \re {\mathrm{re}}
\def \e {\mathrm{e}}
\def \div {\mathrm{div}}
\def \BE {\mathrm{BE}}
\def \SBE {\mathrm{SBE}}
\def \SCN {\mathrm{SCN}}
\def \Sr {\mathrm{S}}
\def \BDF {\mathrm{BDF}}
\def \esssup {\mathrm{ess~sup}}
\def \data {\mathrm{data}}
\def \dist {\mathrm{dist}}
\def \diam {\mathrm{diam}}
\def \supp {\mathrm{supp}}
\def \Rs   {\mathbf{R}_{{\mathrm es}}}
%%%%%%%%%%%%%%%%% mathpzc %%%%%%%%%%%%%%%%%%%%%%%%%%%
\def \k{\mathpzc{k}}
\def \Rp{\mathpzc{R}}
%\def \Jp{\mathpzc{J}}

%%%%%%%%%%%%%%%%%%%%%%%%%%%%%%%%%%%%%%%%

\def \LL{L^{\infty}(L^{2})}
\def \LH{L^{2}(H^{1})}

\def \esssup {\mathrm{ess~sup}}
\def \supp {\mathrm{supp}}

\def \Rb {\mathbf{R}}
\def \Jb {\mathbf{J}}
\def  \apos {\emph{a posteriori~}}
\def  \apri {\emph{a priori~}}
\newcommand{\myblue}{\textcolor[rgb]{0.00,0.00,0.63}}
%\newcommand{\green}{\textcolor[rgb]{0.00,0.25,0.00}}
%%%%%%%%%%%%%%%%%%%%%%

%%%%%%%%%%%%%%%%%%%%%%%%%%%%%%%%%%%%%%%%%%%%%%%%%%%%%%%%%%%%%%%%%%%%%%%%%%%%%%%%%%%%%%%%%
\title{A Priori  Error Bounds  for Parabolic Interface Problems with Measure Data}
\author{Jhuma Sen Gupta\thanks{Corresponding Author: Department of Mathematics, Bits Pilani Hyderabad, Hyderabad - 500078, India
  (\tt{jhumagupta08@gmail.com}).}.}
\date{}
\maketitle
\textbf{ Abstract.}{\small{ This article   studies  a  priori  error analysis  for linear parabolic interface problems with measure data in time in a bounded convex polygonal domain in $\mathbb{R}^2$.   We have used the standard  continuous fitted finite element discretization for the space. Due to the low regularity of the data of the problem, the solution possesses  very low regularity in the entire domain. A  priori error bound in the $L^2(L^2(\Omega))$-norm for the spatially discrete finite element approximations are derived under minimal regularity with the help of the $L^2$ projection operators and  the duality argument. The interfaces are assumed to be smooth for our purpose.}}\\


%\vspace{.1in}

%\textbf{AMS subject classifications.} 65N15, 65N30

\section{Introduction}
%\textbf{ 1.  Abstract.}
The aim of this paper is to study  a priori error analysis of the spatially discrete  finite element approximation for the linear parabolic interface problems with measure data in time.
 %Interface problems arise in  a wide variety of applications in science and engineering,
%such as in material science (cf. Lady{\v{z}}enskaja \emph{et al.} \cite{ladyzensk}),
% fluid dynamics (cf. Reusken and Nguyen \cite{reusken09} and Tu and Peskin \cite{tu92}), electrodynamics (cf. Lumer and Weis \cite{ lumer01})  and so on.
  To begin with,  we first introduce the following parabolic interface problem.

Let $\Omega $ be a bounded convex polygonal domain in $\mathbb{R}^2$ with  Lipschitz boundary $\partial \Omega$, and  let
$\Omega_1$ be a subdomain of $\Omega$ with $C^2$ boundary $\partial \Omega_1:=\Gamma$.
The interface $\Gamma$ divides the domain   $\Omega$
 into two subdomains $\Omega_1$ and $\Omega_2:=\Omega \setminus \Omega_1$. Consider the linear parabolic interface problem of the  form
\be
\partial_t u(x,t) + \mathcal {L} u  = \mu \;\;\; \mbox{in}\;\;\; \Omega_T: = \Omega \times (0,T],\; T>0 \label{main:1}
\ee
with initial and boundary conditions
\be
u(x,0) = u_{0}(x) \;\;\;\mbox{in}\;\;\; \Omega;\;\;\;\;\;\;u=0 \;\;\;\mbox{on} \;\;\; \partial \Omega_T: =\partial\Omega\times(0,T] \label{main:2}
\ee
and jump conditions on the interface
\be
[u] = 0,\;\;\;\left[\beta \frac{\partial u}{\partial\textbf{n}}\right] = 0\;\;\;\mbox{across} \;\;\; \Gamma_T: = \Gamma\times[0,T],\label{main:3}
\ee
where the operator $\mathcal{L}$ is a second order linear elliptic operator defined by
\beno
\mathcal{L} (w) := - \div(\beta(x)\nabla w).
\eeno
The symbol $[v]$ denotes the jump of a quantity $v$ across the interface $\Gamma$,  i.e., $[v](x)$ = $v_{1}(x) - v_{2}(x)$, $x \in\Gamma$ with $v_{i}(x) = v(x)|_{\Omega_{i}}$, $i=1,2$ and $\partial_t =\frac{\partial }{\partial t}$.
 The diffusion coefficient $\beta(x)$ is assumed to be positive and piecewise constant on each subdomain,  i.e.,
$$\beta(x) = \beta_{i} \;\;\;\mbox{for}\;\;\;   x\in \Omega_{i}, \;i=1,2.$$
The symbol $\textbf{n}$ denotes the unit outward normal to the boundary $\partial\Omega_{1}$.

The initial function $u_0(x) \in L^2(\Omega)$ and $\mu = f \sigma$ with $f\in \mathcal{C}([0, T]; L^2(\Omega))$ and $\sigma\in \mathcal{M}[0, T]$, where $\mathcal{M}[0, T]$ denotes the space of the real and regular Borel measures in $[0, T]$. Indeed   $\mathcal{M}[0, T]$ is defined as the dual space of ${\mathcal{C}}[0, T]$ with the standard operator norm given by
\beno
\| \sigma\|_{\mathcal{M}[0, T]} = \sup \left\{\int_0^T g d\sigma \;:\; g\in \mathcal{C}[0, T], \; \sup_{ t \in [0, T]} |g(t)| \le 1\right\}.
\eeno

Further, the interface $\Gamma$ is assumed to be of arbitrary shape but is of class $C^2$ for our purpose.

The parabolic problems with measure data in time mainly appears in the study of  parabolic optimal control problems with pointwise state constraints (cf. \cite{casas1997, gong13,  meidner2011}) and references cited therein. Parabolic problems with measure data arise in wide variety of applications, such as in the modeling of transport equations for  effluent discharge in aquatic media \cite{araya2006}, in the design of disposal of sea outfalls discharging polluting effluent from a sewerage systems \cite{martinez2000}, and many more \cite{droniou2000, gong13, ramos2002}.

The existence and uniqueness of the solutions for both elliptic and parabolic problems with measure data has been investigated by Boccardo and Gallou\"{e}t in \cite{boccardo1989} , Casas in \cite{casas1997}. The finite element method for elliptic problems with measure data have been studied by Araya in \cite{araya2006},
 I. Babuska in \cite{babuska1971}, Casas in \cite{casas1985}, R. Scott in \cite{scott1973, Scott1976} and subsequently, for parabolic problems by Gong in \cite{gong13}.

Finite element method for parabolic interface problems have been extensively studied by many authors, e.g., Chen and Zou \cite {chen-zou98}, Sinha and Deka \cite{sinha05}, Huang and Zou \cite{huang2002}, for details we refer to \cite{bernardi00, bramble96, me2016} and references listed therein.
To the best of the author´s knowledge, finite element method for parabolic interface problems with measure data in time has not been explored yet.
Therefore,  an attempt has been made in this paper to derive a priori error bounds for parabolic interface problems with measure data in time.

For interface problems, it is known that because of the discontinuity of the coefficients along the interface $\Gamma$, the solution of the problem has low global regularity \cite{ladyzensk}. In addition, the solutions of the parabolic problems with measure data  in time also exhibit low regularities \cite{gong13}. Therefore, one follows the duality argument to derive an error bound which in turn depends on a priori error bound for some backward parabolic problem with source term in $L^2(\Omega)$. Therefore, one can use the $H^2$ regularity of the solution of the backward problem (see, [\cite{gong13}, Lemma 2.1]), which is in contrast  to the present case. Indeed we have used the global $H^1$ regularity of the solution to derive a priori bound (see, Lemma \ref{lem:apri:2}) for backward problem (\ref{back:1}). To achieve this, we have derived some new approximation results (see, Lemma \ref{lem: L2Ritz}) for the Ritz- projection operator $R_h$, defined in Section 3.

The structure of this paper is   as follows. In Section \ref{sec2}, we briefly introduce some notations  which will be used throughout, finite element discretizations of the domain and recall some interpolation results  from the literature. In addition, we also recall stability results for parabolic interface problems and we discuss wellposedness of parabolic interface problems with measure data. Section \ref{sec3} is devoted to discuss spatially discrete finite element approximation for parabolic interface problems with measure data and  some related auxiliary results. Finally, we present a priori error bound in the $L^2(L^2)$-norm for interface problems with measure data with some concluding remarks.

Through out this article we use the constant $c$ as the generic one.

%\section{Functional analysis notations and backgrounds.}
\section{Preliminaries} \label{sec2}
In this section, we introduce some standard function spaces including some embedding results, the finite element discretization of the domain $\Omega$.
 We also recall some  approximation properties of the   Lagrange  interpolation operator with low regularity conditions  from \cite{chen-zou98}. In addition, we recall the stability results for parabolic interface problems. And finally, we present the existence and uniqueness of weak solutions of interface problems with measure data in time.
\subsection{Function spaces} We shall use the standard notations for function spaces (see, e.g., \cite{adams, evans}).
Given a Lebesgue measurable set $\mathcal{N} \subset \mathbb{R}^{2}$ and $1\leq p\leq \infty$, $L^{p}(\mathcal{N})$ refers to the standard Lebesgue spaces with the norm $\|\cdot\|_{L^{p}(\mathcal{N})}$. In particular, $L^{2}(\mathcal{N})$ is a Hilbert space  with respect to the norm induced by the inner product $( v, w) = \int _{\mathcal{N}} v(x) w(x) dx $.  We denote the norm of $L^{2}(\mathcal{N})$ by $\|\cdot\|_{\mathcal{N}}$.  For an integer $m > 0$, $H^m(\mathcal{N})$ denotes the usual  Sobolev space with the standard norm $\|\cdot\|_{m,\mathcal{N}}$.
 Further, the function space $H_{0}^{1}(\mathcal{N})$ is a subspace of $H^1(\mathcal{N})$ whose elements have  vanishing trace on the boundary $\partial \mathcal{N}$. For simplicity of notation,  we will skip the subscript $\mathcal{N}$ whenever $\mathcal{N}= \Omega$. We denote $H^{-1} (\Omega)$ as the dual space of $H_0^1(\Omega)$ and the corresponding norm is denoted by $\|\cdot\|_{-1}$.

In addition,  $\mathcal{D}(\Omega_T)$ denotes the space  of $\mathcal{C}_{\infty}(\Omega_T)$ functions with compact support in $\Omega_T$. Moreover,  the $L^2$ - inner product on $L^2(\Omega_T)$, denoted by $ ( v, w )_{\Omega_T}$, defined by
\beno
(v, w)_{\Omega_T} = \int _{\Omega_T} v w dx dt \;\;\;\; \forall v, w \in L^2(\Omega_T).
\eeno

For a real Banach space $\mathbf{B}$ and for $1 \leq p < +\infty$, we define
\beno
L^p(0, T; \mathbf{B}) \;=\;\left\{v : (0,T)\rightarrow \mathbf{B} \;\;\left |\right. \;\; v \text{ is measurable and} \;\;\int_{0}^{T} \|v(t)\|_{\mathbf{B}}^{p} dt  <  + \infty \right\}
\eeno
equipped with the norm
\beno
\|v\|_{L^p(0, T; \mathbf{B})} \;:=\;\left( \int_{0}^{T} \|v(t)\|_{\mathbf{B}}^{p}\right)^{\frac{1}{p}},
\eeno
with the standard modification for $p= \infty$.

We shall also work on the following space:
$X = H_0^1(\Omega)\cap H^2(\Omega_1)\cap H^2(\Omega_2)$ equipped with the norm
\beno
\|v\|_{X} := \|v\|_{H^{1}_{0}(\Omega)} + \|v\|_{H^{2}(\Omega_{1})} + \|v\|_{H^{2}(\Omega_{2})}.
\eeno
 Further, we set
 \beno
 \mathcal{X}(0, T) := L^2(0, T; H_0^1(\Omega))\cap H^1(0, T; H^{-1}(\Omega))
 \eeno
 and
 \beno
 \mathcal{Y}(0, T) := L^2(0, T; X(\Omega))\cap H^1(0, T; L^2(\Omega)).
 \eeno
 It is well-known that $\mathcal{X}(0, T) \hookrightarrow \mathcal{C}([0, T]; L^2(\Omega))$ and $\mathcal{Y}(0, T) \hookrightarrow \mathcal{C}([0, T]; H_0^1(\Omega))$ (see, \cite{Lions, gong13}).

 Now,  we introduce the bilinear forms $a(\cdot, \cdot)$ corresponding to the linear operator $\mathcal{L}$ on $\Omega$ and $\Omega_T$ respectively, as follows:
 \beno
 a(v, w) = \int_{\Omega} \beta \nabla v \nabla w dx \;\;\; \forall v, w\in H_0^1 (\Omega),
 \eeno
 and
 \beno
 a_{T}{(v, w)} = \int_{\Omega_T} \beta \nabla v \nabla w dx dt \;\;\; \forall v, w\in L^2(0, T; H_0^1 (\Omega)).
 \eeno
 We assume that the bilinear form $a(\cdot, \cdot)$ is bounded and coercive on $H_0^1 (\Omega)$, i.e., $\exists  \,\alpha,\gamma > 0 $ such that
\be
|a(v,w)| \leq \alpha\; \|v\|_{1} \|w\|_{1}\;\;\forall v, w \in H_{0}^{1}(\Omega), \label{conts}
\ee
and
\be
a(v,v) \geq \;\gamma \|v\|^{2}_{1}\;\;\forall v \in H_{0}^{1}(\Omega). \label{coercive}
\ee

%%%%%%%%%%%%%%%%%%%%%%%%%%%%%%%%%%%%%%%%%%%%%%%%%%%%%%%%%%%%%%%%%%%%%%%%%%%%%%%%%%%%%%%%
\subsection{Finite Element Discretization of the domain $\Omega$} \label{sec:2.2}
\indent
%Let $\Pc :=\{ (t_{n-1}, t_n]\}_{n=1}^{N}$ be a partition of $(0,T]$ with
%  $I_n :=(t_{n-1}, t_n]$ and  $k:= t_n- t_{n-1} = \frac{T}{N}$ be the constant time step.

To define the finite element approximation we now describe the conforming shape regular
triangulation $\Tc_h = \{K\}$ of $\bar{\Omega}$. To  start with we approximate the domain $\Omega_{1}$ by a polygon $P_{\Omega_1}$ with boundary $\Gamma_{P}$ such that all the vertices of the polygon lie on the interface $\Gamma$. Thus, $\Gamma_P$ divides the domain $\Omega$ into two subsequent subdomains $P_{\Omega_1}$ and $P_{\Omega_2}$, where $P_{\Omega_2}$ is a polygon  approximating the domain $\Omega_2$. We now make the following assumptions on the triangulation $\Tc_h $ (cf. \cite{chen-zou98}):

\begin{description}
\item $\textbf{A1}$. If $K_{1}, K_{2} \in \mathcal{T}_{h}$ and $K_{1} \neq K_{2}$, then either $K_{1}\cap K_{2} = \emptyset$ or $K_{1}\cap K_{2}$ share a common edge or a common vertex. We also assume that each triangle is either in $P_{\Omega_1}$ or in $P_{\Omega_2}$ or intersects the interface $\Gamma$ in at most an edge.

\item $\textbf{A2}$.
Let $h := \max \{ h_K \mid h_K = \text{diam}(K), \; K \in \Tc_h \}$.
The family of conforming
shape regular triangulations $\Tc_h$ are assumed to be quasi-uniform,  i.e., there exist constants $C_{0}, C_{1} > 0 $ independent of $h$ such that
\beno
C_{0} r_{K} \leq h \leq C_{1}\bar{r}_{K} \;\;\;\;\;\;\; \forall K \in \mathcal{T}_{h},
\eeno
where  $r_{K}$ and $\bar{r}_{K}$, respectively be the diameters of the inscribed and circumscribed circles of a triangle $K$.

\end{description}

For a shape regular triangulation $\mathcal{T}_{h}$ of $\Omega$, we consider the following finite element space
\beno
\mathbb{V}_h := \left \{\chi \in H_{0}^{1}(\Omega) \;\;\left|\right. \;\chi|_{K} \in \mathbb{P}_{1}(K) \; \text{for all} \; K \in \mathcal{T}_{h} \right\},
\eeno
where $\mathbb{P}_1(K)$ is the space of polynomials of degree at most $1$ over $K$.

%%%%%%%%%%%%%%%%%%%%%%%%%%%%%%%%%%%%%%%%%%%%%%%%%%%%%%%%%%%%%%%%%%%%%%%%%%%%%%

%\emph{$L^{2}$-projection operator:} The $L^{2}$-projection operator is a map $\Pi_0^{n}: L^{2}(\Omega) \longrightarrow \S^{n}$ such that for $v \in L^{2}(\Omega)$ and $0 \le n \le N$,
% \be
% \langle \Pi_0^{n}\; v\;,\chi_{n}\rangle =\langle v,\;\chi_{n}\rangle \;\;\forall \chi_{n} \in \S^{n}.\label{L2po:bdf}
% \ee

 \subsection{Interpolation Estimates}
The a priori  error analysis for parabolic problems  uses  the approximation properties of the standard  Lagrange  interpolation operator. It is well known that for $O(h^2)$ approximation
results using the piecewise linear finite elements, one requires the global $H^2$ regularity of the function (\cite{ciarlet}). But due to the discontinuity of the coefficient along the interface $\Gamma$, the solution of the parabolic interface problem is only in $H^1(\Omega)$ globally. Hence,  the standard approximation properties do not apply directly for interface problems. Indeed the following
results are true, (cf. \cite{chen-zou98}). Note that,  the approximation results derived in \cite{chen-zou98} uses the global $H^1(\Omega)$ regularity and the estimates are nearly optimal order  up to $|\log h|$ factor.

\begin{lemma}\cite{chen-zou98}\label{lem:lag} Let $\Pi_{h} : \mathcal{C}(\bar{\Omega}) \longrightarrow V_h$ be the standard Lagrange interpolation operator \cite{ciarlet}. Then the following
interpolation estimates hold:
\beno
\| v - \Pi_{h} v\| + h \| \nabla(v - \Pi_{h} v)\| \le c h^2 |\log h|^{\frac 1 2 } \| v\|_X, \;\; \forall v \in X.\label{int -lag}
\eeno
\end{lemma}

\subsection{Stability results for parabolic interface problems} \label{sec:2.2}
\indent
To discuss the solution of the (\ref{main:1}) - (\ref{main:3}) in the weak sense we consider the forward and backward in time parabolic interface  problems of the following form: For $g\in L^2(\Omega_T)$, let $\phi$ and $\psi$ be the solutions of 

\bea\label{ford:1}
\begin{cases}
\partial_t \phi(x,t) + \mathcal {L} \phi   =  g \;\;\; \mbox{in}\;\;\; \Omega_T \label{main-f:1}   \\
\phi(x,0) =  0 \;\;\;\mbox{in}\;\;\; \Omega;\;\;\;\;\;\;\phi =0 \;\;\;\mbox{on} \;\;\; \partial \Omega_T \label{main-f:2} \\
{[\phi]} = 0,\;\;\;\left[\beta \frac{\partial \phi}{\partial\textbf{n}}\right] = 0\;\;\;\mbox{across} \;\;\; \Gamma_T,\label{main-f:3}
\end{cases}
\eea

and
\bea\label{back:1}
\begin{cases}
 - \partial_t \psi(x,t) + \mathcal {L^*} \psi   =  g \;\;\; \mbox{in}\;\;\; \Omega_T \label{main-b:1}   \\
\psi(x,T) =  0 \;\;\;\mbox{in}\;\;\; \Omega;\;\;\;\;\;\;\psi =0 \;\;\;\mbox{on} \;\;\; \partial \Omega_T \label{main-b:2} \\
{[\psi]} = 0,\;\;\;\left[\beta \frac{\partial \psi}{\partial\textbf{n}}\right] = 0\;\;\;\mbox{across} \;\;\; \Gamma_T,\label{main-b:3}
\end{cases}
\eea
respectively, where, $\mathcal{L}^*$ denotes the  adjoint of the elliptic operator $\mathcal{L}$, given by
 \beno
 \mathcal{L}^* w = - \sum_{i,j =1}^2 \partial_{x_j}(\beta({x}) \partial x_i w).
 \eeno

Next, we present the stability results for parabolic interface problems (cf. \cite{chen-zou98, ladyzensk, sinha05}).

\begin{lemma}\label{lem:regu}
Let $g \in H^1(0,T; L^2(\Omega))$. Then the problems (\ref{ford:1}) and  (\ref{back:1}) have unique solution $v$ ($v = \phi$ or $v = \psi$) such that $v\in L^2(0, T; X) \cap H^1(0, T; L^2(\Omega))$. Moreover, $L^2(0, T; X) \cap H^1(0, T; L^2(\Omega)) \hookrightarrow \mathcal {C}([0, T]; H^1(\Omega))$ and $v$ satisfies the following a priori estimates
\be
\|v\|_{L^2(0, T; X)} + \|v_t\|_{L^2(0, T; L^2(\Omega))} \le c \|g\|_{L^2(0, T; L^2(\Omega))},  \label{stab:f}
\ee
and
\beno \label{stab:1}
\|\phi(T)\|_1 \le c \|g\|_{L^2(0, T; L^2(\Omega))}, \;\;\;\;\; \|\psi(0)\|_1 \le c \|g\|_{L^2(0, T; L^2(\Omega))}.
\eeno
\end{lemma}

\subsection{Existence and uniqueness of weak solution of interface problem with measure data}
In this section, we will discuss the existence and uniqueness of solutions of parabolic interface problem (\ref{main:1}) - (\ref{main:3})  with measure data in time.

The notion of the  solution of the problem (\ref{main:1}) - (\ref{main:3}) in the weak sense can be generalized by the well-known transposition method which in turn based on the solutions of a forward and backward in time parabolic problems (\cite{gong13, Lions}).
The next result presents the existence and uniqueness of the solution of the problem (\ref{main:1}) - (\ref{main:3}) by means of transposition technique.
\begin{lemma}
Let $f\in \mathcal {C}([0, T]; L^2(\Omega))$ and $\sigma \in \mathcal{M}[0, T]$. Then with the assumption that $v(x, T) =0$, the problem (\ref{main:1}) - (\ref{main:3}) has a unique solution $u\in L^2(0, T; H_0^1(\Omega)) \cap L^{\infty}(0, T; L^2(\Omega))$ such that
\be
-(u, \partial_t v)_{\Omega_T} + a_T(u, v) = \langle \mu, v\rangle_{\Omega_T} + (u_0, v(x, 0)) \quad \forall v\in \mathcal{X}(0, T), \label{weak}
\ee
where
\beno
 \langle\mu, v\rangle_{\Omega_T} = \int_{\bar{\Omega}_T} v d\mu = \int_0^T\left(\int_{\Omega} f(x, t) v(x, t) dx\right) d\sigma(t) \quad \forall v\in \mathcal {C}([0, T]; L^2(\Omega)).
\eeno
In addition, the following regularity result holds:
\be
\|u\|_{L^2(0, T; H_0^1(\Omega))} + \| u\|_{L^{\infty}(0, T; L^2(\Omega))} \le c \left(\|f\|_{L^{\infty}(0, T; L^2(\Omega))}  \| \sigma\|_{\mathcal{M}[0, T]} + \|u_0\|\right). \label{stab:m}
\ee
\end{lemma}
Following the argument of \cite{gong13}, proof of the lemma follows.

%%%%%%%%%%%%%%%%%%%%%%%%%%%%%%%%%%%%%%%%%%%%%%%%%%%%%%%%%%%%%%%%%%%%%%%%%%%%%%%%
\section{A priori error analysis}\label{sec3}
%\subsection{Finite element approximation of parabolic interface problems with measure data in time}
 To discuss a priori error analysis, we first define spatially discrete approximation for the problem (\ref{main:1}) - (\ref{main:3}) which in turn depends on the weak formulation of the problem (\ref{weak}).

The spatially discrete approximation for the problem (\ref{main:1}) - (\ref{main:3}) reads as: Find $u_h \in {L^{2}(0, T; \mathbb{V}_h)}$ with $v_h(x, T) = 0$ such that
\be
-(u_h, \partial_t v_h)_{\Omega_T} + a_T(u_h, v_h) = \langle \mu, v_h\rangle_{\Omega_T} + (\pi_0 u_0, v_h(x, 0)) \quad \forall v_h\in H^1(0, T; \mathbb{V}_h), \label{semi}
\ee
where $\pi_0$ is a suitable chosen projection operator from $L^2(\Omega)$ into $\mathbb{V}_h$.

In the above,
\beno
 \langle\mu, v_h\rangle_{\Omega_T} = \int_{\bar{\Omega}_T} v_h d\mu = \int_0^T\left(\int_{\Omega} f(x, t) v_h(x, t) dx\right) d\sigma(t) \quad \forall v_h \in \mathcal {C}([0, T]; \mathbb{V}_h).
\eeno

\subsection{Some auxiliary results}
To derive a priori error bound,  we first introduce some projection operators.

\emph{$L^2$- projection operator}. The $L^2$- projection operator is the operator $L_h: L^2(\Omega)  \longrightarrow \mathbb{V}_h$ such that for $w\in L^2(\Omega)$,
\be
(L_h w, v_h) = (w, v_h) \quad \forall v_h\in \mathbb{V}_h. \label{L2:proj}
\ee

\emph{Ritz - projection operator}.  The Ritz - projection operator is the operator $R_h: H_0^1(\Omega)  \longrightarrow \mathbb{V}_h$ such that for $w\in H_0^1(\Omega)$,
\be
a(R_h w, v_h) = a(w, v_h) \quad \forall v_h\in \mathbb{V}_h. \label{Ritz:proj}
\ee

The following lemma gives the approximation results for both the $L^2$- projection  and  the Ritz - projection operator.

\begin{lemma}\label{lem: L2Ritz}
Let $L_h$ be the $L^2$- projection operator and $R_h$ be the Ritz - projection operator defined by (\ref{L2:proj}) and (\ref{Ritz:proj}), respectively. Then the  following approximation results hold:
\bea
\| w - L_h w\|_{-1} + h \| w - L_h w\| &\le& c h^2 \| w\|_1 \label{approx:L2} \\
\| w - R_h w\| + h |\log h|^{\frac 1 2}\| w - R_h w\|_1 &\le& c h^2 |\log h| \| w\|_X. \label{approx:Ritz}
\eea
\end{lemma}

\begin{proof}
We refer to \cite{gong13, rannacher82} for the  proof of the estimate (\ref{approx:L2}). The proof of  estimate (\ref{approx:Ritz}) uses the standard  trick. But, for clarity of presentation we provide here the brief explanation.

It is well-known that the Ritz - projection is stable in the $H^1$-norm (see, \cite{rannacher82}), that is
\be
\|R_h w\|_1 \le c \|w\|_1. \label{ritz:stabi}
\ee
Invoke triangle inequality and  note that $R_h (\Pi_h w) = \Pi_h w$,      we obtain
\bea
\| w - R_h w\|_1 &\le& \| w  - \Pi_h w\|_1 + \|R_h w  - \Pi_h w\|_1  \nonumber\\
                  &\le& \| w  - \Pi_h w\|_1 + c \| w  - \Pi_h w\|_1,
\eea
where we have used equation (\ref{ritz:stabi}). The proof of the estimate follows from  the Lemma \ref{lem:lag}.

To obtain the $L^2$-norm estimate, we use the usual duality trick. For any $w\in X$, let $v\in H_0^1(\Omega)$ be the  unique solution of the following elliptic interface problem
\bea
a(v, z ) &=& \left( w- R_h w, z\right) \;\;\forall z \in H^{1}_{0}(\Omega),  \label{est:R:1} \\
v  & = & 0 \;\; \text{on}\;\; \partial\Omega, \nonumber\\
  \left[v \right] &= &  0,  \;\; \left[\beta \frac{\partial v}{\partial \textbf{n}}\right] =0 \;\; \text{across}\;\;  \Gamma \nonumber.
\eea
Further, the solution $v$ satisfies  the following regularity result (cf. \cite{chen-zou98})
\be
\|v\|_{X} \leq c\;  \|w- R_h w\|.           \label{er-2:cn:2}
\ee
Equations (\ref{est:R:1}) and (\ref{er-2:cn:2}),  together with (\ref{Ritz:proj}), leads to
\beano
\|w- R_h w\|^2  & = & a(v - R_h v, w - R_h w) + a(R_h v, w - R_h w) \nonumber\\
                      &\le& c \|v - R_h v\|_1 \|w - R_h w\|_1\nonumber\\
                      & \le & c h |\log h|^{\frac 1 2} \|v\|_{X} \|w - R_h w\|_1 \nonumber\\
                       & \le & c h |\log h|^{\frac 1 2} \|w- R_h w\|    \|w - R_h w\|_1,  \label{er-2:cn:3}\eeano

 Invoking the Lemma \ref{lem: L2Ritz},  the desired estimate follows and this completes  the proof.
\end{proof}

Now, we define spatially discrete finite element approximation for the backward parabolic interface problem (\ref{back:1}) which is required in our subsequent analysis.

In order to define this, we first write the weak formulation of the interface problem (\ref{back:1}) as: Seek $\psi \in {H^{1}(0, T; H_0^1(\Omega))}$  such that

\be
-(\partial_t \psi, v )_{\Omega_T} + a_T{(\psi, v)} = (g, v)_{\Omega_T} \quad \forall v\in H^1(0, T; H_0^1(\Omega)), \label{weak:back}
\ee
with $\psi (\cdot, T) =0$.

Then the spatially discrete finite element approximation to (\ref{back:1}) is stated as follows: Find $\psi_h \in {H^{1}(0, T; \mathbb{V}_h)}$ such that
\be \label{semi:back}
-(\partial_t \psi_h, v_h)_{\Omega_T} + a_T{(\psi_h, v_h)} = (g, v_h)_{\Omega_T} \quad \forall v_h\in H^1(0, T; \mathbb{V}_h)
\ee
with $\psi_h (\cdot, T) =0$.

The next lemma presents a priori error bound for the backward parabolic interface problem (\ref{back:1}).
\begin{lemma}\label{lem:apri:1}
Let $\psi \in \mathcal{Y}(0, T)\hookrightarrow \mathcal{C}([0, T]; H_0^1(\Omega))$ and $\psi_h$ be the solutions of the problem (\ref{back:1}) and (\ref{semi:back}), respectively. If $\psi_h(0) = L_h \psi(0)$, then we have
\be\label{back:apriori}
\| \psi(t)  -  \psi_h(t) \|^2 + \int_0^t \|\psi(s)  -  \psi_h(s)\|_1^2 ds
 \le c h^2 \left( \|\partial_t \psi\|^2_{L^{2}(0, T; L^{2}(\Omega))} + |\log h| \|\psi\|^2_{L^{2}(0, T; X)}  \right).
\ee
\end{lemma}

Since  we need to derive the error bound under the low regularity assumption stated in the above lemma,  we shall make use of the $L^2$ projection operator instead of the Ritz-projection operator which is in contrast to the usual a priori  error analysis for parabolic problems, see \cite{thomee06} for details. We will follow the approach of \cite{hou2002} to derive error bound   (\ref{back:apriori}).

\begin{proof}
Subtracting (\ref{semi:back}) from (\ref{weak:back}) we obtain
\be
-(\partial_t \psi(t)  - \partial_t \psi_h(t), v_h) + a(\psi(t) - \psi_h(t), v_h) =0 \quad \forall v_h\in \mathbb{V}_h, \quad \text{a.e.} \;t\in(0, T]. \label{back:pf:1}
\ee
Set $v_h(t) = L_h \psi(t)$ in (\ref{back:pf:1}) we have for a.e. $t\in(0, T]$,
\bea \label{back:pf:2}
\frac 1 2 \frac{d}{dt} \|\psi(t)  -  \psi_h(t)\|^2 &+& a (\psi(t)  -  \psi_h(t), \psi(t)  -  \psi_h(t)) \nonumber\\
                  &=& (\partial_t \psi(t)  -  \partial_t \psi_h(t), \psi(t) - v_h(t)) + a(\psi(t) - \psi_h(t), \psi(t) - v_h(t)) \nonumber\\
                  %&=& (\partial_t \psi(t)  -  \partial_t {\psi_h}(t), \psi(t) - L_h \psi(t)) + a(\psi(t) - \psi_h(t), \psi(t) - L_h \psi(t)) \nonumber\\
                  &=& (\partial_t \psi(t)  -  \partial_t  {L_h}{ \psi}(t), \psi(t) - L_h \psi(t)) +  \nonumber\\
                  &&  (\partial_t  L_h \psi(t) - \partial_t {\psi_h}(t), \psi(t) - L_h \psi(t)) + a(\psi(t) - \psi_h(t), \psi(t) - L_h \psi(t)) \nonumber\\
                  &=& (\partial_t \psi(t)  -  \partial_t  {L_h}{ \psi}(t), \psi(t) - L_h \psi(t)) + a(\psi(t) - \psi_h(t), \psi(t) - L_h \psi(t)) \nonumber\\
                  &=& \frac 1 2 \frac{d}{ dt} \|{ \psi}(t) - L_h \psi(t)\|^2  + a(\psi(t) - \psi_h(t), \psi(t) - L_h \psi(t)),
                  \eea
where we have used the definition (\ref{L2:proj}) and the fact that $\partial_t  ({L_h}{ \psi}(t) -  \psi_h(t))  \in\mathbb{V}_h$.

Integrate (\ref{back:pf:2}) from $0$ to $t$ and with an aid of the Cauchy-Schwarz inequality, a simple calculation implies
\beano \label{back:pf:3}
&& \| \psi(t)  -  \psi_h(t) \|^2 + \gamma \int_0^t \|\psi(s)  -  \psi_h(s)\|_1^2 ds \nonumber\\
     && \le c\left( \| \psi(t) - L_h \psi(t)\|^2 + \| \psi(0) - \psi_h(0)\|^2 - \|\psi(0) - L_h\psi(0)\|^2 + \int_0^t \|\psi(s) - L_h\psi(s)\|_1^2 ds \right).
\eeano
If $\psi_h(0) = L_h \psi(0)$ and using the fact that $\| w - L_h w\|_1 \le c \| w - R_h w\|_1$ $\forall w \in H_0^1(\Omega)$ (see, \cite{hou2002}), it follows that
\bea\label{back:pf:4}
 \| \psi(t)  -  \psi_h(t) \|^2 &+& \gamma \int_0^t \|\psi(s)  -  \psi_h(s)\|_1^2 ds \nonumber\\
      &\le& c\left( \| \psi(t) - L_h \psi(t)\|^2  + \int_0^t \|\psi(s) - R_h\psi(s)\|_1^2 ds \right)\nonumber\\
      & \le& c\left(\max_{t\in[0, T]}\| \psi(t) - L_h \psi(t)\|^2  + \int_0^t \|\psi(s) - R_h\psi(s)\|_1^2 ds \right).
\eea
Again,  in view of the embedding result $\mathcal{X}(0, T)\hookrightarrow \mathcal{C}([0, T]; L^2(\Omega))$ (cf. \cite{evans, hou2002}), we have
\be\label{back:pf:5}
\max_{t\in[0, T]}\|w(t)\| \le c\left( \| w\|_{L^2(0, T; H_0^1(\Omega))} +  \| \partial_t w\|_{L^2(0, T; H^{-1}(\Omega))}\right).
\ee
With an aid of Lemma \ref{lem: L2Ritz} together with (\ref{back:pf:4}) and (\ref{back:pf:5}) we obtain
\beno\label{back:pf:6}
\| \psi(t)  -  \psi_h(t) \|^2 + \gamma \int_0^t \|\psi(s)  -  \psi_h(s)\|_1^2 ds
 \le c h^2 \left( \|\partial_t \psi\|^2_{L^{2}(0, T; L^{2}(\Omega))} + |\log h| \|\psi\|^2_{L^{2}(0, T; X)} \right).
\eeno
\end{proof}

Indeed,  we have obtained the following result.
\begin{lemma}\label{lem:apri:2}
Let $\psi \in \mathcal{Y}(0, T)\hookrightarrow \mathcal{C}([0, T]; H_0^1(\Omega))$ and $\psi_h$ be the solutions of the problem (\ref{back:1}) and (\ref{semi:back}), respectively. If $\psi_h(0) = L_h \psi(0)$, then we have
\be\label{back:apri:1}
\|\psi  -  \psi_h \|_{L^{\infty}(0, T; L^{2}(\Omega))} \le c h \left( \|\partial_t \psi\|_{L^{2}(0, T; L^{2}(\Omega))} + |\log h|^{\frac 1 2} \|\psi\|_{L^{2}(0, T; X)} \right).
\ee
\end{lemma}
Now we are ready to present the main theorem of this article which provides a priori error bound in the $L^2(L^2)$-norm for the parabolic interface problem of the form (\ref{main:1}) - (\ref{main:3}).
\begin{theorem}
Let $u$ and $u_h$ be the  solutions of the problem (\ref{weak}) and (\ref{semi}), respectively. Then, for $f\in \mathcal{C}([0, T]; L^2(\Omega))$, $\sigma \in\mathcal{M}[0, T]$ and $u_0\in L^2(\Omega)$,  we have the following estimate
\be\label{thm:main}
\| u - u_h\|_{L^2(0, T; L^2(\Omega))} \le c h \max\{1, |\log h|^{\frac 1 2}\}  \left\{ \|u_0\| + \|f\|_{L^{\infty}(0, T; L^2(\Omega))} \| \sigma\|_{\mathcal{M}[0, T]}\right\}.
\ee
\end{theorem}
We follow the approach of Wei \cite{hou2002} to prove a priori error bound (\ref{thm:main}). The main idea in the proof is to employ duality  argument.
\begin{proof}
Subtracting (\ref{weak}) from (\ref{semi}) and using (\ref{L2:proj}), we obtain the orthogonality relation as
\be\label{pf:main:1}
((u - u_h), - \partial_t v_h)_{\Omega_T} + a_T(u - u_h, v_h) = 0 \quad \forall v_h\in H_1(0, T, \mathbb{V}_h).
\ee
In order to obtain an error bound of the form (\ref{thm:main}), we calculate the error $u - u_h$ in the $L^2(L^2)$-norm using the dual norm as
\be\label{pf:main:2}
\| u - u_h\|_{L^2(0, T; L^2(\Omega))} = \sup\left\{{{\frac{(g, u - u_h)_{\Omega_T}}{\|g\|_{L^2(0, T; L^2(\Omega))}}} \; :\; g\in L^2(0, T; L^2(\Omega)), \; g\neq 0}\right\}.
\ee
Therefore, for $g\in L^2(0, T; L^2(\Omega))$, equation (\ref{weak}), together with (\ref{semi}) and (\ref{pf:main:1}), leads to
\bea\label{pf:main:3}
(g, u - u_h)_{\Omega_T} &=& (-\partial_t \psi + \mathcal{L}^{*} \psi, u - u_h)_{\Omega_T} \nonumber\\
                      &=& (u, -\partial_t \psi)_{\Omega_T} + {a_T} (u,\psi)  + (u_h, \partial_t \psi)_{\Omega_T} -  {a_T} (u_h,\psi) \nonumber\\
                      &=& (u, -\partial_t \psi)_{\Omega_T} + {a_T} (u,\psi)  +  (u_h, \partial_t \psi_h)_{\Omega_T} -  {a_T} (u_h,\psi_h) \nonumber\\
                      &=& \langle \mu, \psi\rangle_{\Omega_T} + (u_0, \psi(x, 0)) - \langle \mu, \psi_h \rangle_{\Omega_T} - (L_h u_0, \psi_h(x, 0))\nonumber\\
                      &=& \langle \mu, \psi - \psi_h\rangle_{\Omega_T} + (u_0, \psi(0) - \psi_h(0))\nonumber\\
                      &=& \int_0^T\left(\int_{\Omega} f(x, t) (\psi - \psi_h) dx\right) d\sigma(t) + (u_0, \psi(0) - \psi_h(0))\nonumber\\
                      &\le& c \left\{\| f\|_{L^{\infty}(0, T; L^{2}(\Omega))} \|\sigma\|_{\mathcal{M}[0, T]}  + \|u_0\| \right\}\| \psi - \psi_h\|_{L^{\infty}(0, T; L^{2}(\Omega))}.
                      \eea
An application of Lemma \ref{lem:apri:2} together with Lemma \ref{lem:regu} leads to
\beano\label{pf:main:4}
(g, u - u_h)_{\Omega_T} &\le& c h \left(\| f\|_{L^{\infty}(0, T; L^{2}(\Omega))} \|\sigma\|_{\mathcal{M}[0, T]}  + \|u_0\| \right) \left( \| \partial_t \psi\|_{L^{2}(0, T; L^{2}(\Omega))} + |\log h|^{\frac 1 2} \| \psi\|_{L^{2}(0, T; X)}\right) \nonumber\\
                       &\le& c h \max\{1, |\log h|^{\frac 1 2}\}\left(\| f\|_{L^{\infty}(0, T; L^{2}(\Omega))} \|\sigma\|_{\mathcal{M}[0, T]}  + \|u_0\| \right) \|g\|_{L^{2}(0, T; L^{2}(\Omega))}.
\eeano
The rest of the proof follows from (\ref{pf:main:2}).
\end{proof}

\section{Conclusion and extension}
This article investigates a priori error analysis for parabolic interface problems with measure data in time in a bounded convex domain in $\mathbb{R}^2$. We have only considered spatially discrete approximation in the error analysis. It is interesting to extend these results for fully discrete approximations for parabolic interface problems with measure data in time. In addition, we will also extend this idea for parabolic interface problems with measure data in space. We remark that such an extension is not straightforward because of additional regularity issues in space  for interface problems. We will also address computational issues for such kind of problems in future.

%%%%%%%%%%%%%%%%%%%%%%%%%%%%%% bibliography starts here %%%%%%%%%%%%%%%%%%%%%%%%%
 \providecommand{\bysame}{\leavevmode\hbox
to3em{\hrulefill}\thinspace}

\bibliographystyle{siam}{\small
\bibliography{measure}}

\def\cprime{$'$} \def\cprime{$'$}
\begin{thebibliography}{10}

\bibitem{adams}
{\sc R.~A. Adams and J.~J.~F. Fournier}, {\em Sobolev Spaces}, vol.~140 of Pure
  and Applied Mathematics, Elsevier, Amsterdam, second~ed., 2003.

\bibitem{araya2006}
{\sc R.~Araya, E.~Behrens, and R.~Rodr\'{\i}guez}, {\em A posteriori error
  estimates for elliptic problems with {D}irac delta source terms}, Numer.
  Math., 105 (2006), pp.~193--216.

\bibitem{babuska1971}
{\sc I.~Babu\v{s}ka}, {\em Error-bounds for finite element method}, Numer.
  Math., 16 (1970/71), pp.~322--333.

\bibitem{bernardi00}
{\sc C.~Bernardi and R.~Verf{\"u}rth}, {\em Adaptive finite element methods for
  elliptic equations with non-smooth coefficients}, Numer. Math., 85 (2000),
  pp.~579--608.

\bibitem{boccardo1989}
{\sc L.~Boccardo and T.~Gallou\"{e}t}, {\em Nonlinear elliptic and parabolic
  equations involving measure data}, J. Funct. Anal., 87 (1989), pp.~149--169.

\bibitem{bramble96}
{\sc J.~H. Bramble and J.~T. King}, {\em A finite element method for interface
  problems in domains with smooth boundaries and interfaces}, Adv. Comput.
  Math., 6 (1996), pp.~109--138.

\bibitem{casas1985}
{\sc E.~Casas}, {\em {$L^2$} estimates for the finite element method for the
  {D}irichlet problem with singular data}, Numer. Math., 47 (1985),
  pp.~627--632.

\bibitem{casas1997}
\leavevmode\vrule height 2pt depth -1.6pt width 23pt, {\em Pontryagin's
  principle for state-constrained boundary control problems of semilinear
  parabolic equations}, SIAM J. Control Optim., 35 (1997), pp.~1297--1327.

\bibitem{chen-zou98}
{\sc Z.~Chen and J.~Zou}, {\em Finite element methods and their convergence for
  elliptic and parabolic interface problems}, Numer. Math., 79 (1998),
  pp.~175--202.

\bibitem{hou2002}
{\sc K.~Chrysafinos and L.~S. Hou}, {\em Error estimates for semidiscrete
  finite element approximations of linear and semilinear parabolic equations
  under minimal regularity assumptions}, SIAM J. Numer. Anal., 40 (2002),
  pp.~282--306.

\bibitem{ciarlet}
{\sc P.~G. Ciarlet}, {\em The Finite Element Method for Elliptic Problems},
  vol.~40 of Classics in Applied Mathematics, SIAM, Philadelphia, 2002.

\bibitem{droniou2000}
{\sc J.~Droniou and J.-P. Raymond}, {\em Optimal pointwise control of
  semilinear parabolic equations}, Nonlinear Anal., 39 (2000), pp.~135--156.

\bibitem{evans}
{\sc L.~C. Evans}, {\em Partial differential equations}, vol.~19 of Graduate
  Studies in Mathematics, American Mathematical Society, Providence, RI,
  second~ed., 2010.

\bibitem{gong13}
{\sc W.~Gong}, {\em Error estimates for finite element approximations of
  parabolic equations with measure data}, Math. Comp., 82 (2013), pp.~69--98.

\bibitem{huang2002}
{\sc J.~Huang and J.~Zou}, {\em Some new a priori estimates for second-order
  elliptic and parabolic interface problems}, J. Differential Equations, 184
  (2002), pp.~570--586.

\bibitem{ladyzensk}
{\sc O.~A. Lady{\v{z}}enskaja, V.~A. Solonnikov, and N.~N. Ural{\cprime}ceva},
  {\em Linear and Quasilinear Equations of Parabolic Type}, Translated from the
  Russian by S. Smith. Translations of Mathematical Monographs, Vol. 23,
  American Mathematical Society, 1968.

\bibitem{Lions}
{\sc J.-L. Lions and E.~Magenes}, {\em Non-homogeneous boundary value problems
  and applications. {V}ol. {I}}, Die Grundlehren der mathematischen
  Wissenschaften, Band 181, Springer-Verlag, New York-Heidelberg, 1972.
\newblock Translated from the French by P. Kenneth.

\bibitem{martinez2000}
{\sc A.~Mart\'{\i}nez, C.~Rodr\'{\i}guez, and M.~E. V\'{a}zquez-M\'{e}ndez},
  {\em Theoretical and numerical analysis of an optimal control problem related
  to wastewater treatment}, SIAM J. Control Optim., 38 (2000), pp.~1534--1553.

\bibitem{meidner2011}
{\sc D.~Meidner, R.~Rannacher, and B.~Vexler}, {\em A priori error estimates
  for finite element discretizations of parabolic optimization problems with
  pointwise state constraints in time}, SIAM J. Control Optim., 49 (2011),
  pp.~1961--1997.

\bibitem{ramos2002}
{\sc A.~M. Ramos, R.~Glowinski, and J.~Periaux}, {\em Pointwise control of the
  {B}urgers equation and related {N}ash equilibrium problems: computational
  approach}, J. Optim. Theory Appl., 112 (2002), pp.~499--516.

\bibitem{rannacher82}
{\sc R.~Rannacher and R.~Scott}, {\em Some optimal error estimates for
  piecewise linear finite element approximations}, Math. Comp., 38 (1982),
  pp.~437--445.

\bibitem{scott1973}
{\sc R.~Scott}, {\em Finite element convergence for singular data}, Numer.
  Math., 21 (1973/74), pp.~317--327.

\bibitem{Scott1976}
\leavevmode\vrule height 2pt depth -1.6pt width 23pt, {\em Optimal {$L^{\infty
  }$} estimates for the finite element method on irregular meshes}, Math.
  Comp., 30 (1976), pp.~681--697.

\bibitem{me2016}
{\sc J.~Sen~Gupta, R.~K. Sinha, G.~M.~M. Reddy, and J.~Jain}, {\em A posteriori
  error analysis of two-step backward differentiation formula finite element
  approximation for parabolic interface problems}, J. Sci. Comput., 69 (2016),
  pp.~406--429.

\bibitem{sinha05}
{\sc R.~K. Sinha and B.~Deka}, {\em Optimal error estimates for linear
  parabolic problems with discontinuous coefficients}, SIAM J. Numer. Anal., 43
  (2005), pp.~733--749.

\bibitem{thomee06}
{\sc V.~Thom{\'e}e}, {\em Galerkin Finite Element Methods for Parabolic
  Problems}, vol.~25 of Springer Series in Computational Mathematics,
  Springer-Verlag, Berlin, second~ed., 2006.

\end{thebibliography}
\end{document}